\documentclass[a4paper]{amsart}
\usepackage[latin1]{inputenc}
\usepackage{amssymb}
\usepackage{amsmath}
\usepackage{mathrsfs}
\usepackage{eufrak}
\usepackage{amsthm}
\usepackage{amsfonts}
\usepackage{textcomp}
\usepackage{graphicx}
\usepackage[pdftex]{color}
\usepackage{paralist}
\usepackage[shortlabels]{enumitem}
\usepackage{hyperref}
\usepackage{comment}
\usepackage[arrow, matrix, curve]{xy}

\newtheorem*{corollary*}{Corollary}
\newtheorem{theorem}{Theorem}[section]

\newtheorem{corollary}[theorem]{Corollary}
\newtheorem{lemma}[theorem]{Lemma}
\newtheorem{proposition}[theorem]{Proposition}
\newtheorem{question}[theorem]{Question}

\newtheorem*{claim*}{Claim}
\newtheorem*{conjecture}{Conjecture}

\theoremstyle{definition}
\newtheorem{definition}[theorem]{Definition}
\newtheorem*{theorem*}{Theorem}
\newtheorem*{proposition*}{Proposition}

\newtheorem{example}[theorem]{Example}

\theoremstyle{remark}

\numberwithin{equation}{theorem}

\makeatletter
\renewcommand*\env@matrix[1][\
arraystretch]{%
  \edef\arraystretch{#1}%
  \hskip -\arraycolsep
  \let\@ifnextchar\new@ifnextchar
  \array{*\c@MaxMatrixCols c}}
\makeatother

\begin{document}

\title{Finitistic Auslander algebras}
\date{\today}



\author{Ren\'{e} Marczinzik}
\address{Institute of algebra and number theory, University of Stuttgart, Pfaffenwaldring 57, 70569 Stuttgart, Germany}
\email{marczire@mathematik.uni-stuttgart.de}

\begin{abstract}
Recently, Chen and Koenig in \cite{CheKoe} and Iyama and Solberg in \cite{IyaSol} independently introduced and characterised algebras with dominant dimension coinciding with the Gorenstein dimension and both dimensions being larger than or equal to two. In \cite{IyaSol}, such algebras are named Auslander-Gorenstein algebras. Auslander-Gorenstein algebras generalise the well known class of higher Auslander algebras, where the dominant dimension additionally coincides with the global dimension. In this article we generalise Auslander-Gorenstein algebras further to algebras having the property that the dominant dimension coincides with the finitistic dimension and both dimension are at least two. We call such algebras finitistic Auslander algebras. As an application we can specialise to reobtain known results about Auslander-Gorenstein algebras and higher Auslander algebras such as the higher Auslander correspondence with a very short proof.
We then give several conjectures and classes of examples for finitistic Auslander algebras.
For a local Hopf algebra $A$ and an indecomposable non-projective $A$-module $M$, we show that $End_A(A \oplus M)$ is always a finitistic Auslander algebra of dominant dimension two. In particular this shows that $Ext_A^1(M,M)$ is always non-zero, which generalises a result of Tachikawa who proved that $Ext_A^1(M,M) \neq 0$ for indecomposable non-projective modules $M$ over group algebras of $p$-groups.
We furthermore conjecture that every algebra of dominant dimension at least two which has exactly one projective non-injective indecomposable module is a finitistic Auslander algebra. We prove this conjecture for a large class of algebras which includes all representation-finite algebras.
\end{abstract}

\maketitle
\section*{Introduction}
Let an algebra always be a finite dimensional connected algebra over a field $K$, which is not semi-simple. All modules are finite dimensional right modules if nothing is stated otherwise.
In this article we generalise Auslander-Gorenstein algebras introduced in \cite{IyaSol} as algebras having dominant dimension equal to the Gorenstein dimension and both dimensions being larger than or equal to two (in fact Iyama and Solberg also include selfinjective algebras in their definition of Auslander-Gorenstein algebras, but we do not include selfinjective algebras here as the developed theory for selfinjective algebras is trivial). Auslander-Gorenstein algebras contain the important class of higher Auslander algebras, introduced in \cite{Iya}. Recall that the finitistic dimension of an algebra is defined as the supremum of projective dimensions of all modules having finite projective dimension. It is a major open problem in the representation theory of finite dimensional algebras whether the finitistic dimension is always finite.
Note that in case an algebra is Gorenstein, the Gorenstein dimension equals the finitistic dimension (see for example \cite{Che}). Thus algebras with finitistic dimension equal to the dominant dimension, which is larger than or equal to two, generalise Auslander-Gorenstein algebras. We call such algebras \emph{finitistic Auslander algebras} and deduce some of their properties, including a generalisation of the celebrated higher Auslander correspondence for finite dimensional algebras first proven in \cite{Iya}. We characterise finitistic Auslander algebras in terms of Gorenstein homological algebra and the category of modules having a certain dominant dimension. Let $Dom_d(A)$ denote the full subcategory of modules having dominant dimension at least $d$, $Gp(A)$ the subcategory of Gorenstein projective modules and $Gp_{\infty}(A)$ the full subcategory of modules having infinite Gorenstein projective dimension, $Proj(A)$ the full subcategory of modules which are projective and $Proj_{\infty}(A)$ the full subcategory of modules having infinite projective dimension. We refer to the preliminaries for more information and definitions.
\begin{theorem*} 
Let $A \cong End_B(M)$ be an algebra of dominant dimension $d \geq 2$, where $M$ is a generator-cogenerator of $mod-B$.
The following are equivalent:
\begin{enumerate}
\item $A$ is a finitistic Auslander algebra.
\item $Dom_d(A) \subseteq Proj(A) \cup Proj_{\infty}(A)$.
\item $Dom_d(A) \subseteq Gp(A) \cup Gp_{\infty}(A)$.
\item $add(M)-resdim(X)= \infty$ for all $X \in M^{\perp d-2} \setminus add(M)$.
\end{enumerate}
\end{theorem*}
Thus generator-cogenerators $M$ with $add(M)-resdim(X)= \infty$ for all $X \in M^{\perp d-2} \setminus add(M)$ generalise the classical cluster tilting objects introduced in \cite{Iya} and the precluster tilting objects introduced in \cite{IyaSol}.

In particular, specialising our results to finite Gorenstein or finite global dimension, we obtain quick proofs of some known facts such as the higher Auslander correspondence relating higher Auslander algebras and cluster tilting objects or the classification of Gorenstein projective modules over Auslander-Gorenstein algebras.

The rest of the articles is guided by conjectures that we motivate and prove in special cases.

\begin{conjecture} 
Let $A$ be a local selfinjective algebra and $M$ an indecomposable non-projective $A$-module.
Then $End_A(A \oplus M)$ is a finitistic Auslander algebra of dominant dimension 2.
\end{conjecture}

We prove this conjecture in a special case:
\begin{theorem*}
Let $A$ be a local Hopf algebra and $M$ an indecomposable non-projective $A$-module.
Then $End_A(A \oplus M)$ is a finitistic Auslander algebra of dominant dimension 2.
\end{theorem*}
This theorem implies that $Ext_A^1(M,M) \neq 0$ for $M$ as in the theorem. This generalises an old theorem of Tachikawa who proved this for group algebras of $p$-groups. We also include an example of a local Hopf algebra that is not isomorphic to a group algebra of a $p$-group to show that we give a proper generalisation of the theorem of Tachikawa.
We then give some other examples that motivate this conjecture on local selfinjective algebras and also show that in general the class of finitistic Auslander algebras is much larger than the class of Auslander-Gorenstein algebras. For example we show that for a local selfinjective algebra $A$ with simple module $S$, the algebra $B=End_A(A \oplus S)$ is a standardly stratified finitistic Auslander algebra of dominant dimension 2. This algebra $B$ is an Auslander-Gorenstein algebra iff $A \cong K[x]/(x^n)$ for some $n \geq 2$ and a higher Auslander algebra iff $A \cong K[x]/(x^2)$.

For some of our theory we can include algebras that have finitistic dimension equal to the dominant dimension, even when the dominant dimension is equal to one. We call an algebra with finitistic dimension equal to the non-zero dominant dimension \emph{weak finitistic Auslander algebras}.
Our next conjecture is related to the finitistic dimension conjecture as we will see later. The truth of the following conjecture would give a large and easy construction of finitistic Auslander algebras.

\begin{conjecture}
Let $A$ be an algebra of dominant dimension at least one that has exactly one indecomposable projective non-injective $A$-module. Then $A$ is a weak finitistic Auslander algebra.
\end{conjecture}
Note that the class of algebras with dominant dimension at least one that have at exactly one indecomposable projective non-injective $A$-module is very large, for example it generalises the class of proper almost selfinjective algebras from \cite{FHK}.
Let $P^{< \infty}(A)$ denote the full subcategory of modules having finite projective dimension. Note that in a representation-finite algebra all subcategories are contravariantly finite and thus the next theorem applies to all representation-finite algebras.
\begin{theorem*}
Let $A$ be a finite dimensional algebra such that $P^{< \infty}(A)$ is contravariantly-finite. Assume furthermore that $A$ has dominant dimension at least one and exactly one indecomposable projective non-injective module. Then $A$ is a weak finitistic Auslander algebra.

\end{theorem*}

The examples and results in this article motivate the following conjecture:
\begin{conjecture}
Let $n \geq 2$.
There exists a polynomial function $f(n)$ such that the following is true: \newline
Every connected non-selfinjective algebra with $n$ simple modules that has dominant dimension at least $f(n)$ is a finitistic Auslander algebra.
\end{conjecture}
We think that one might choose a polynomial function with $n \leq f(n) \leq 2n$ for each $n \geq 2$. In particular, the author is not aware of a non-selfinjective algebra with dominant dimension $\geq 2$ and having two simple modules that is not a finitistic Auslander algebra.
The author thanks Jeremy Rickard for allowing him to use the theorem \ref{Ricktheo} in this article. This answered a question of the author raised in mathoverflow, see http://mathoverflow.net/questions/257744/finite-addn-resolution. The author thanks Matthew Pressland for helpful discussions on shifted tilting modules that lead to an improvement of \ref{propocotilting}. The author thanks Xingting Wang for suggesting the example in \ref{hopfexample} and he thanks Zhao Tiwei for useful comments.
Many results in this article were tested with the GAP-package QPA and the author is thankful to the QPA-team for their work, see \cite{QPA}.
\section{Preliminaries}
Throughout $A$ is a finite dimensional and connected algebra over a field $K$. Furthermore, we assume that $A$ is not semisimple. We always work with finite dimensional right modules, if not stated otherwise. By $mod-A$, we denote the category of finite dimensional right $A$-modules. 
For background on representation theory of finite dimensional algebras and their homological algebra, we refer to \cite{ASS} or \cite{SkoYam}.
For a module $M$, $add(M)$ denotes the full subcategory of $mod-A$ consisting of direct summands of $M^n$ for some $n \geq 1$.
A module $M$ is called \emph{basic} in case $M \cong M_1 \oplus M_2 \oplus ... \oplus M_n$, where every $M_i$ is indecomposable and $M_i$ is not isomorphic to $M_j$ for $i \neq j$. The \emph{basic version} of a module $N$ is the unique (up to isomorphim) module $M$ such that $add(M)=add(N)$ and such that $M$ is basic. An algebra is called basic in case the regular module is basic.
We denote by $S_i=e_iA/e_iJ$, $P_i=e_i A$ and $I_i=D(Ae_i)$  the simple, indecomposable projective and indecomposable injective module, respectively, corresponding to the primitive idempotent $e_i$. \newline
The \emph{dominant dimension} domdim($M$) of a module $M$ with a minimal injective resolution \newline $(I_i): 0 \rightarrow M \rightarrow I_0 \rightarrow I_1 \rightarrow ...$ is defined as: \newline
domdim($M$):=$\sup \{ n | I_i $ is projective for $i=0,1,...,n \}$+1, if $I_0$ is projective, and \newline domdim($M$):=0, if $I_0$ is not projective. \newline
The \emph{codominant dimension} of a module $M$ is defined as the dominant dimension of the $A^{op}$-module $D(M)$.
The dominant dimension of a finite dimensional algebra is defined as the dominant dimension of the regular module. It can be shown that the dominant dimension of an algebra always equals the dominant dimension of the opposite algebra, see for example \cite{Ta}.
So domdim($A$)$ \geq 1$ means that the injective hull of the regular module $A$ is projective or equivalently, that there exists an idempotent $e$ such that $eA$ is a minimal faithful projective-injective module.
Algebras with dominant dimension larger than or equal to 1 are called QF-3 algebras.
For more information on dominant dimensions and QF-3 algebras, we refer to \cite{Ta}.
An algebra $A$ is called \emph{Gorenstein} in case $Gdim(A):=injdim(A)$ equals $projdim(D(A))< \infty$. In this case $Gdim(A)$ is called the \emph{Gorenstein dimension} of $A$ and we say that $A$ has infinite Gorenstein dimension if $injdim(A)= \infty$ or $projdim(D(A))= \infty$. Note that $Gdim(A)= max \{ injdim(e_iA) | e_i$ a primitive idempotent$ \}$ and $domdim(A)= min \{ domdim(e_iA) | e_i $ a primitive idempotent $ \}$.
We denote by $Proj(A)$ the full subcategory of projective modules and by $Proj_{\infty}(A)$ the full subcategory of modules of infinite projective dimension. $Dom_d(A)$ denotes the full subcategory of modules having dominant dimension at least $d$.
The Morita-Tachikawa correspondence (see for example \cite{Ta}) says that an algebra $A$ has dominant dimension at least two iff $A \cong End_B(M)$ for some generator-cogenerator $M$ of $mod-B$ and some algebra $B$ that is then isomorphic to $eAe$, when $eA$ is a minimal faithful projective-injective $A$-module. Mueller's theorem says that in this case the dominant dimension of $A$ equals $\inf \{ i \geq 1 | Ext_B^{i}(M,M) \neq 0 \} +1$, see \cite{Mue}.
We will need the following results that can be viewed as refinements of results of Mueller. The theorem can be found in \cite{Mar} as theorem 2.2. with detailed references to the article \cite{APT}.
\begin{theorem}
\label{ARSmaintheorem}
Let $A$ be an algebra of dominant dimension at least two with minimal faithful projective-injective left module $P$ and minimal faithful projective-injective right module $I$. Let $B=End_A(P)$. We have $B \cong End_A(I)$.
\begin{enumerate}
\item $F:=Hom_A(P,-) : Dom_2(A) \rightarrow mod-B$ is an equivalence of categories.
$F$ restricts to an equivalence between add($I$) and the category of injective $B$-modules.
\item The functor $G:=Hom_{B}(P,-) : mod-B \rightarrow Dom_2(A)$ is inverse to $F$.
\item For $i \geq 3$, $F$ restricts to an equivalence $F: Dom_i(A) \rightarrow (P)^{\perp i-2}$, where $P$ is viewed as a $B$-module.
\end{enumerate}
\end{theorem}
An algebra $A$ is called \emph{higher Auslander algebra} in case $\infty>domdim(A)=gldim(A) \geq 2$, see \cite{Iya} and $A$ is called \emph{Auslander-Gorenstein algebra} in case $\infty>domdim(A)=Gdim(A) \geq 2$, see \cite{IyaSol} (here we exclude selfinjective algebras that are not interesting for the theory we develope).
A module $M$ is called \emph{Gorenstein projective} in case $Ext^{i}(D(A),\tau(M)) \cong Ext^{i}(M,A) \cong 0$ for all $i \geq 1$. Every non-projective Gorenstein projective module has infinite projective dimension. As in the case of usual projective resolutions, every module $M$ has a resolution by (possibly infinitely generated) Gorenstein projective modules and a corresponding \emph{Gorenstein projective dimension} $Gpd(M)$, see \cite{Che} for more details. $\Omega^{i}(A-mod)$ denotes the full subcategory of all projective modules and modules which are $i$-th syzygies, $Gp(A)$ denotes the full subcategory of Gorenstein projective modules and $Gp_{\infty}(A)$ denotes the full subcategory of modules having infinite Gorenstein projective dimension. 
We will need the following proposition:
\begin{proposition}\label{marvil}
Let $A$ be an algebra of dominant dimension $d \geq 1$, then $\Omega^{i}(A-mod)=Dom_i(A)$ for every $i \leq d$.
\end{proposition}
\begin{proof}
See \cite{MarVil}, proposition 4.
\end{proof}
For a given subcategory $C$ of $mod-A$, a \emph{minimal right $C$-approximation} of a module $X$ is a right minimal map $f: N \rightarrow X$ with $N \in C$ such that $Hom(L,f)$ is surjective for every $L \in C$. Such minimal right approximations always exist and are unique up to isomorphism in case $C=add(M)$ for some module $M$. In case $C=add(M)$, one defines $\Omega_M^{0}(X):=X$, $\Omega_M^{1}(X)$ as the kernel of such an $f$ and inductively $\Omega_M^{n}(X):=\Omega_M^{n-1}(\Omega_M^{1}(X))$. One then defines $add(M)-resdim(X):= inf \{ n \geq 0 | \Omega_M^{n}(X) \in add(M) \}$. Dually, one can define minimal left $C$-approximations. Given an algebra $A$, which is isomorphic to $End_B(M)$ for some algebra $B$ with generator-cogenerator $M$, one can show that minimal $add(M)$-resolutions in $mod-B$ of a module $X$ correspond to minimal projective resolutions of the module $Hom_B(M,X)$ in $mod-A$. See \cite{CheKoe} section 2.1. for more information on this.
A subcategory $C$ of $mod-A$ is called \emph{contravariantly finite} in case every module $X \in mod-A$ has a minimal right $C$-approximation. A subcategory $C$ is called \emph{resolving} in case it contains the projective modules, is closed under extensions and closed under kernels of surjections. 
We will need the following result, that can be found in \cite{AR}, 3.9.:
\begin{proposition}\label{ARpropo}
Let $C$ be a resolving contravariantly finite subcategory of $mod-A$. Then every module in $C$ has finite projective dimension bounded by $t$ in case all of the modules $X_i$ have finite projective dimension bounded by $t$, where $f_i: X_i \rightarrow S_i$ are minimal right $C$-approximations of the simple modules $S_i$.
\end{proposition} 
For a module $M$, we define $M^{\perp n}:= \{ X \in mod-A | Ext^{i}(M,X)=0$ for all $i=1,...,n \}$.
The \emph{finitistic dimension} of an algebra is defined as $findim(A)= \sup \{ pd(N) | pd(N) < \infty \}$.
The \emph{global Gorenstein projective dimension} of an algebra is defined as the supremum of all Gorenstein projective dimensions of modules. It is known that the global Gorenstein projective dimension is finite iff the algbra is Gorenstein, see \cite{Che} corollary 3.2.6. The \emph{finitistic Gorenstein projective dimension} is defined as $Gfindim(A)= \sup \{ Gpd(N) | Gpd(N) < \infty \}$ and in \cite{Che} one finds a quick proof that this always equals the usual finitistic dimension in theorem 3.2.7. We call a module $M$ $d$-rigid, in case $Ext^{i}(M,M)=0$ for $i=1,2,...,d$.
We will also need the following theorem, which can be found as theorem 3.2.5. in \cite{Che} and can be used as a characterisation of the Gorenstein projective dimension of a module.
\begin{theorem} \label{gordimchara}
Let $M$ be a module. $M$ has finite Gorenstein projective dimension at most $n$ iff in every exact sequence of the form $0 \rightarrow K \rightarrow G_{n-1} \rightarrow ... \rightarrow G_1 \rightarrow G_0 \rightarrow M \rightarrow 0$ with Gorenstein projective modules $G_i$, also the module $K$ is Gorenstein projective.

\end{theorem}

Recall that the eveloping algebra $A^{e}$ for an arbitrary algebra $A$ is defined as $A^{e}:=A^{op} \otimes_K A$ and an algebra is called $m$-periodic in case the $A^{e}-$module $A$ has $\Omega$-period $m$. Being $m$-periodic implies that the algebra is selfinjective and that every indecomposable non-projective module $M$ is periodic of period at most $m$, that is $\Omega^i(M) \cong M$ for some $i$ with $1 \leq i \leq m$. See chapter IV.11. of \cite{SkoYam} for this and more on periodic algebras.

\section{Finitistic Auslander algebras}
This section introduces finitistic Auslander algebras and gives new relations between dominant dimension and the finitistic dimension.
\begin{lemma} \label{findimlemma}
Let $A$ be an algebra of dominant dimension $d \geq 1$. \newline
We have $findim(A)=d+\sup \{ pd(N) | domdim(N) \geq d , pd(N) < \infty \}$.
\end{lemma}
\begin{proof}
First note that the finitistic dimension is larger than or equal to the dominant dimension:
The exact sequence coming from a minimal injective coresolution of the regular module: 
$0 \rightarrow A \rightarrow I_0 \rightarrow \cdots \rightarrow I_{d-1} \rightarrow \Omega^{-d}(A) \rightarrow 0$ shows that the module $\Omega^{-d}(A)$ has finite projective dimension $d$. Thus the finitistic dimension is at least $d$.
Assume the projective dimension $s \geq d$ is attained at the module $X$: $pd(X)=s$.
Looking at the minimal projective resolution of $X$: $0 \rightarrow P_s \rightarrow \cdots P_0 \rightarrow X \rightarrow 0$ and using $pd(X)=pd(\Omega^{d}(X))+d$ one immediatly obtains the lemma, since $\Omega^{d}(X)$ has dominant dimension at least $d$ by \ref{marvil} and its projective dimension equals $s-d$.
\end{proof}
Let $P^{< \infty}(A)$ denote the full subcategory of modules having finite projective dimension. It is well known that the finitistic dimension of an algebra $A$ is finite in case $P^{< \infty}(A)$ is contravariantly finite, for example using \ref{ARpropo}. Here we show that it is enough that a smaller subcategory $Dom_i(A) \cap Proj_{<\infty}$ is contravariantly finite in case the algebra has dominant dimension at least $d$ and $i \leq d$.
\begin{proposition}
Let $A$ be an algebra of positive dominant dimension $d$, then $A$ has finite finitistic dimension in case the subcategory $Dom_l(A) \cap Proj_{<\infty}(A)$ is contravariantly finite for some $l \leq d$.
\end{proposition}
\begin{proof}
Let $C:=Dom_l(A) \cap P^{< \infty}(A)$ for some $l \leq d$.
We want to use \ref{ARpropo}. First note that the intersection of two resolving subcategories is resolving and that $Dom_l(A)$ is resolving (see for example \cite{MarVil} proposition 1), while the property that $P^{< \infty}(A)$ is resolving is well known. Thus $C$ is a contravariantly finite resolving subcategory and the $X_i$ (defined as the modules, such that $f_i: X_i \rightarrow S_i$ are minimal right $C$-approximations of the simples) have finite projective dimension bounded by some number $t$, since they are contained in $P^{< \infty}(A)$. Thus all modules in $C$ have finite projective dimension bounded by $t$ and the result follows from \ref{findimlemma}.
\end{proof}
The proposition is useful in various situations where it is hard to calculate $Proj_{<\infty}(A)$ but the subcategory $Dom_d(A)$ is representation-finite.
For example for the large class of monomial algebras $A$ of dominant dimension at least two, we have $Dom_2(A)= \Omega^2(mod-A)$ and this category is representation-finite for monomial algebras (see \cite{Z}) and thus also the subcategory $Dom_2(A) \cap Proj_{<\infty}(A)$ is representation-finite and we can conclude directly that such algebras have finite finitistic dimension and can calculate the finitistic dimension by calculating approximations of the simple modules in the subcategory $Dom_2(A) \cap Proj_{<\infty}(A)$, that is usually much smaller than the subcategory $Proj_{<\infty}(A)$. 
We remark that we are not aware of an algebra with dominant dimension $d \geq 1$ such that such that $Dom_l(A) \cap Proj_{<\infty}(A)$ is not contravariantly finite for any $0 \leq l \leq d$.
We formulate this as a question:
\begin{question}
Given an algebra of dominant dimension $d \geq 1$, is $Dom_l(A) \cap Proj_{<\infty}(A)$ contravariantly finite for some $l$ with $0 \leq l \leq d$?
\end{question}
A positive answer to the previous question would prove the finitistic dimension conjecture for algebras with positive dominant dimension and thus prove the Nakayama conjecture.
Now we come to the generalisation of Auslander-Gorenstein algebras:
\begin{definition}
An algebra with finite dominant dimension $d \geq 2$ is called a \emph{finitistic Auslander algebra} in case its finitistic dimension equals its dominant dimension.
\end{definition}
Note that by the Morita-Tachikawa correspondence, every finitistic Auslander algebra $A$ is isomorphic to an algebra of the form $End_B(M)$ for some algebra $B$ with generator-cogenerator $M$, since by assumption $A$ has dominant dimension at least two.
We remark that every Auslander-Gorenstein algebra and thus every higher Auslander algebra is a finitistic Auslander algebra, since the finitistic dimension equals the Gorenstein dimension in case the Gorenstein dimension is finite (see for example \cite{Che}). We will later see many examples of a finitistic Auslander algebra of infinite Gorenstein dimension, showing that the class of finitistic Auslander algebras is much bigger than the class of Auslander-Gorenstein algebras.

The next theorem gives another characterisation of finitistic Auslander algebras using the subcategory of modules having dominant dimension at least $d$.

\begin{theorem} \label{maintheorem}
Let $A \cong End_B(M)$ be an algebra of finite dominant dimension $d \geq 2$, where $M$ is a generator-cogenerator.
The following are equivalent:
\begin{enumerate}
\item $A$ is a finitistic Auslander algebra.
\item $Dom_d(A) \subseteq Proj(A) \cup Proj_{\infty}(A)$.
\item $Dom_d(A) \subseteq Gp(A) \cup Gp_{\infty}(A)$.
\item $add(M)-resdim(X)= \infty$ for all $X \in M^{\perp d-2} \setminus add(M)$
\end{enumerate}
\end{theorem}
\begin{proof}
First we show that (1) and (2) are equivalent: Just note that by \ref{findimlemma}, $A$ is a finitistic Auslander algebra iff every module of dominant dimension at least $d$ has infinite projective dimension or is projective.
Assume now (1), that is the finitistic dimension of the algebra equals the dominant dimension. Assume $X \in Dom_d(A)$ and $X$ having finite and non-zero Gorenstein projective dimension $s$. Then there exists the following exact sequence, where the left side comes from a minimal Gorenstein projective resolution and the right side comes from a minimal injective coresolution: $0 \rightarrow G_s \rightarrow \cdots \rightarrow G_0 \rightarrow X \rightarrow I_0 \rightarrow \cdots \rightarrow \Omega^{-d}(X) \rightarrow 0 $. This shows that the module $\Omega^{-d}(X)$ has finite Gorenstein projective dimension $s+d>s$ using that $X$ is not Gorenstein projective, by \ref{gordimchara}.

This contradicts the fact that the finitistic Gorenstein projective dimension equals the finitistic dimension which is equal to $s$. This shows that $(1)$ implies $(3)$. \newline
Now assume (3), that is $Dom_d(A) \subseteq Gp(A) \cup Gp_{\infty}(A)$. We use \ref{findimlemma} and show that $\sup \{ pd(N) | domdim(N) \geq d , pd(N) < \infty \}=0$. But this is obvious since every non-projective module in $Dom_d(A) \subseteq Gp(A) \cup Gp_{\infty}(A)$ has infinite projective dimension (recall that Gorenstein projective modules are projective or have infinite projective dimension). This shows that $(3)$ implies $(1)$. \newline
Now we show that $(4)$ is equivalent to $(1)$: \newline
Assume $A$ has dominant dimension $d \geq 2$.
By \ref{findimlemma}, the finitistic dimension equals the dominant dimension iff every non-projective module of dominant dimension at least $d$ has infinite projective dimension. This translates into the condition $add(M)-resdim(X)= \infty$ for all $X \in M^{\perp d-2} \setminus add(M)$ since $add(M)$ resolutions correspond to minimal projective resolutions in $A$ and the subcategory $Dom_d(A)$ without the projectives corresponds to $M^{\perp d-2} \setminus add(M)$ by (3) of \ref{ARSmaintheorem}.
\end{proof}

The next lemma was also noted in \cite{Mar}.
\begin{lemma}\label{gpdomdim}
Let $A$ be an algebra of dominant dimension $d \geq 1$, then every Gorenstein projective module has dominant dimension at least $d$.
\end{lemma}
\begin{proof}
By definition every Gorenstein projective module is in $\Omega^{i}(A-mod)$ for every $i \geq 1$. Now $\Omega^{d}(A-mod)=Dom_d(A)$ by \ref{marvil} and thus $Gp(A) \subseteq \Omega^{d}(A-mod)=Dom_d(A)$.
\end{proof}

Note that in the next proposition, (2) contains the higher Auslander correspondence from \cite{Iya}, where generator-cogenerators with the condition $add(M)=M^{\perp d-2}$ are called cluster tilting objects. We give a very quick proof of the higher Auslander correspondence in (2) but refer to \cite{CheKoe} or \cite{IyaSol} for the second equivalence in (1).
\begin{proposition} \label{correspondences}
Let $B$ an algebra with a generator-cogenerator $M$ and $A=End_B(M)$. Assume $A$ has finite dominant dimension $d \geq 2$, which by Mueller's theorem is equivalent to $M$ being $d-2$ rigid and not $d-1$ rigid.
\begin{enumerate}
\item $A$ is an Auslander-Gorenstein algebra iff $Dom_d(A)=Gp(A)$ iff \newline $add(M)=add(\tau(\Omega^{d-2}(M \oplus D(A))).$
\item $A$ is a higher Auslander algebra iff $Dom_d(A)=Proj(A)$ iff $add(M)=M^{\perp d-2}$.

\end{enumerate}
\end{proposition}
\begin{proof}
\begin{enumerate}
\item Being an Auslander-Gorenstein algebra is equivalent to being a finitistic Auslander algebra and additionally having finite Gorenstein dimension. An algebra is Gorenstein iff it has finite global Gorenstein dimension and thus iff $Gp_{\infty}(A)$ is empty. But by \ref{gpdomdim} $Gp(A) \subseteq Dom_d(A)$ and thus it is an Auslander-Gorenstein algebra iff $Dom_d(A)=Gp(A)$, using (4) of \ref{maintheorem}. For the second equivalence, see \cite{CheKoe} corollary 3.18.
\item Recall that an algebra has finite global dimension iff it is Gorenstein and every Gorenstein projective module is projective, see for example \cite{Che}. Thus the first equivalence follows by the first equivalence in (1). Now let $Af$ be the minimal faithful projective-injective left $A$-module. Then the functor $(-)f$  is an equivalence between $Proj(A)$ and $add(M)$ and between $Dom_d(A)$ and $M^{\perp d-2}$ by \ref{ARSmaintheorem} and this shows the second equivalence. 
\end{enumerate}

\end{proof}

We explicitly state the case $d=2$ since here finitistic Auslander algebras generalise the well known Auslander algebras.
\begin{corollary}
Let $B$ be an algebra with generator-cogenerator $M$ and $A=End_B(M)$. Then $A$ is a finitistic Auslander algebra with finitistic dimension two iff $Ext^{1}(M,M) \neq 0$ and $add(M)-resdim(X)= \infty$ for all $X \in mod-B \setminus add(M)$.
\end{corollary}
\begin{proof}
By Mueller's theorem $Ext^{1}(M,M) \neq 0$ implies that $A$ has dominant dimension $d=2$ and by \ref{maintheorem} (4) the result follows by noting that $M^{\perp 0}=mod-B$.
\end{proof}

The next question is motivated by the characterisation $Dom_d(A)=Gp(A)$ for Auslander-Gorenstein algebras.
\begin{question}
Let $A$ be a finitistic Auslander algebra. Can the subcategory of Gorenstein projective $A$-modules be explicitly described in terms of other subcategories?

\end{question}

\section{Interlude on Hopf algebras}
Before we construct large classes of finitistic Auslander algebras in the next section, we prove several results about local Hopf algebras in this section that we will use. We assume that the reader is familiar with the basics on finite dimensional Hopf algebras over a field $K$ as explained for example in the last chapter of the book \cite{SkoYam}. Recall that we assume that all algebras are non-semisimple unless stated otherwise.

We need several results on Hopf algebras that we quote in the following from the literature.

\begin{theorem}\label{hopflemmas}
For a finite dimensional Hopf algebra $A$ and $A$-modules $M_1$, $M_2$ and $M_3$, then the following holds:
\begin{enumerate}
\item $Ext_A^{i}(M_1 \otimes_K M_2 , M_3) \cong Ext_A^{i}(M_1,Hom_K(M_2,M_3))$, for every $i \geq 1$.
\item $Hom_A(M_1,M_2) \cong M_1^{*} \otimes_K M_2$
\item $M_1$ is projective iff $M_1 \otimes_K M_1^{*}$ is projective.
\item $A$ is selfinjective.
\end{enumerate}
\end{theorem}
\begin{proof}
\begin{enumerate}
\item See \cite{SkoYam}, theorem 6.4. for $i$=0 and for $i>0$ the proof is as in proposition 3.1.8. (ii) of \cite{Ben}.
\item See \cite{SkoYam}, chapter VI. exercise 24.
\item See \cite{SkoYam}, chapter VI. exercise 27.
\item See \cite{SkoYam}, theorem 3.6.
\end{enumerate}
\end{proof}

\begin{proposition} \label{extcrit}
The following are equivalent for two modules $X,Y$ over a local Hopf algebra $A$:
\begin{enumerate}
\item $Ext_A^{1}(X,Y)=0$
\item $Hom_K(X,Y)$ is projective.
\end{enumerate}
\end{proposition}
\begin{proof}
Using (1) and (2) of the above \ref{hopflemmas} we have:
$Ext_A^{1}(X,Y) \cong Ext_A^{1}(K \otimes_K X , Y) \cong Ext_A^{1}(K,Hom_K(X,Y)).$
Now since $K$ is the unique simple modules of the local selfinjective algebra, $Ext_A^{1}(K,Hom_K(X,Y))=0$ iff $Hom_K(X,Y)$ is projective.

\end{proof}

\begin{theorem} \label{tachtheo}
For a local finite dimensional Hopf algebra $A$ we have $Ext_A^1(M,M) \neq 0$ for each non-projective module $M$. 
\end{theorem}
\begin{proof}

By \ref{extcrit}, $Ext^{1}(M,M)=0$ for a module iff $Hom_k(M,M) \cong M^{*} \otimes_k M$ is projective. By \ref{hopflemmas} (3) this is true iff $M$ is projective.
\end{proof}
While the proof of the previous theorem might appear easy and short, recall that we had to use several non-trivial theorem from \ref{hopflemmas}.
In \cite{Ta} theorem 8.6., Tachikawa proved the previous theorem for the special case of $p$-groups. The next corollary is an immediate consequence of Mueller's theorem.
\begin{corollary}
Let $B=End_A(M)$, where $A$ is a local Hopf algebra and $M$ a non-projective generator of $mod-A$, then $B$ has dominant dimension equal to two.
\end{corollary}

The next proposition shows that \ref{extcrit} can also be used to show that $Ext^{1}(X,Y) \neq 0$ for every indecomposable non-projective modules $X,Y$ in certain local Hopf algebras. First we need a lemma:
\begin{lemma} \label{kxlemma}
Let $A=K[x]/(x^n)$ for some $n \geq 2$. Then $Ext_A^{1}(X,Y) \neq 0$ for arbitrary indecomposable non-projective modules $X,Y$.

\end{lemma}
\begin{proof}
This is elementary to check. See for example the preliminaries of \cite{ChMar} for the calculation of $Ext^1$ in symmetric Nakayama algebras.
\end{proof}
We also need the following theorem, see \cite{Ch}:
\begin{theorem} \label{chtheorem}
Let $K$ be a field of characteristic $p$ and $G$ be a finite group such that $p$ divdes the group order. Then a $KG$-module is projective iff it is free on restriction to all the elementary abelian $p$-subgroups of $G$.
\end{theorem}
The next proposition is a slight generalisation of an example by Jeremy Rickard given in http://mathoverflow.net/questions/259344/classification-of-certain-selfinjective-algebras.
\begin{proposition}
Let $K$ be a field of characteristic $p$ and $G$ be a $p$-group having only one non-trivial elementary abelian subgroup $Z$. Let $A=KG$,
then $Ext_A^{1}(X,Y) \neq 0$ for arbitary indecomposable non-projective modules $X,Y$. 
\end{proposition}
\begin{proof}
By \ref{extcrit}, we have to show that $Hom_K(X,Y)$ is never projective for given $X,Y$. Now since $X$ and $Y$ are assumed to be non-projective, their restrictions to $KZ$ is not projective. Thus as $KZ$-modules $X \cong M_1 \oplus N_1$ and $Y \cong M_2 \oplus N_2$ for some indecomposable non-projective $KZ$-modules $M_1$ and $M_2$. Now note that $KZ$ is isomorphic to some algebra of the form $K[x]/(x^n)$ for some $n \geq 2$. Using \ref{kxlemma}, one has that $Hom_K(X,Y)$ is not projective as a $KZ$-module and thus $Hom_K(X,Y)$ is not projective as a $KG$-module. This gives the proposition using \ref{chtheorem}.
\end{proof}
The previous proposition applies for example to the quaternion group over a field of characteristic 2.

To show that \ref{tachtheo} is really a generalisation of the result of Tachikawa, one has to find a finite dimensional local Hopf algebra that is not isomorphic to a group algebra. Xingting Wang suggested to try example (A5) from theorem 1.1 in the paper \cite{NWW}. Here we sketch the proof that it is not isomorphic to a group algebra by calculating the quiver with relations isomorphic to the algebra and then calculating the beginning of a minimal projective resolution of the simple module. The reader can skip this example as it will not be used later.
\begin{example}\label{hopfexample}
Fix an algebraically closed field $K$ of characteristic 2. The algebra $A$ is defined as $K<x,y,z>/(x^2,y^2,xy-yx,xz-zy,yz-zy-x,z^2-xy)$ (for the Hopf algebra structure see \cite{NWW}). 
Note first that $A$ is local of dimension 8 over the field $K$ with basis $\{1,x,y,z,z^2,xz,yz,zy\}$ and the Jacobson radical is the ideal generated by $x,y$ and $z$. Now we have to calculate the second power $J^2$ of the Jacobson radical: It contains $x$, since $x=yz-zy \in J^2$. Since $A$ is not commutative, its quiver can not have just one loop. Thus the dimension of $J^2$ is at most 5. It is clear that $J^2$ contains every basis element expect possibly $y$ and $z$. Thus, since the dimension of $J^2$ is at most 5, $J^2$ has basis $x,z^2,xz,yz,zy$. We will now show that the quiver algebra of $A$ is isomorphic to $K<a,b>/(a^2,b^2-aba)$. Clearly $K<a,b>$ maps onto $A$ by a map $f$, with $f(a)=y$ and $f(b)=z$. Note that $(a^2,b^2-aba)$ is contained in the kernel of $f$, since $y^2=0$ and $z^2-yzy=z^2-(zy+x)y=z^2-xy=0$. Thus there is a surjective map $\hat{f}:k<a,b>/(a^2,b^2-aba) \rightarrow A$ induced by $f$. But since $K<a,b>/(a^2,b^2-aba)$ also has dimension 8, that is in fact an isomorphism. \newline
Now we show that $A=K<a,b>/(a^2,b^2-aba)$ (we will identify $A$ in the following with $K<a,b>/(a^2,b^2-aba)$)  is not isomorphic to a group algebra. Since $A$ has dimension 8 and is not commutative, there are only 2 candidates of group algebras, that could be isomorphic to $A$: The group algebra of the dihedral group of order 8 and the group algebra of the Quaternion group.
Let $S_1$ be the simple $A$ module, then it is elementary to check that $\Omega^{4}(S_1)$ has dimension 9.
The dimension of $\Omega^{4}(S_1)$ is the crucial information that we need to distinguish $A$ from group algebras of dimension 8 over the field. Let $B$ be the group algebra of the dihedral of order 8 over $K$ with simple module $S_2$. Then by \cite{Ben2} chapter 5.13., $\Omega^{4}(S_2)$ has dimension 17 and thus $A$ is not isomorphic to $B$. Let $C$ be the group algebra of the quaternion group of order 8 over $K$ with simple module $S_3$. Then this algebra is 4-periodic (see for example \cite{Erd}) and thus $\Omega^{4}(S_3) \cong S_3$ and $A$ is not isomorphic to $C$. This shows that $A$ is not isomorphic to any group algebra.

\end{example}

By looking at local Hopf algebras, the author noted that there seems to be no known example of a local Hopf algebra that is not a symmetric algebra.
We pose this as a question:
\begin{question}
Is every local Hopf algebra a symmetric algebra?
\end{question}

\section{Examples of finitistic Auslander algebras}
In this section we construct several examples of finitistic Auslander algebras using different methods.
\subsection{Finitistic Auslander algebras from local Hopf algebras}
First we use the results on local Hopf algebras from the previous section to construct finitistic Auslander algebras.
The next lemma is due to Jeremy Rickard.
\begin{lemma}
Let $A$ be a local selfinjective algebra and $M$ an indecomposable module and $\alpha : M^m \rightarrow M^n$ with $n,m >0$ a map between direct sums of $M$ all of whose components are radical maps. Let $F$ be an additive functor such that $F( \alpha)$ is injective, then $F(M)=0$. 
\end{lemma}
\begin{proof}
We can assume that $n$ is a multiple of $m$ by possibly adding extra summands to $M^n$. Write $n=dm$ for some integer $d$. We then have maps $M^{d^s m} \rightarrow M^{d^{s+1} m}$ for $s \geq 0$ by taking direct sums of the map $\alpha$. This gives a sequence of maps $M^m \rightarrow M^{dm} \rightarrow M^{d^2 m} \rightarrow ... \rightarrow M^{d^k m}$, all of which become injective when applying $F$, since $F$ is additive. But choosing $k$ greater than the Loewy length of $End_A(M)$, the composition of this sequence of maps is zero and thus $F(M^m)=0$, giving also $F(M)=0$ using that $F$ is addtive.
\end{proof}

The next theorem is due to Jeremy Rickard.
\begin{theorem} \label{Ricktheo}
Let $A$ be a local selfinjective algebra and let $M$ be an indecomposable nonprojective module with $Ext^{1}(M,M) \neq 0$. Then $B:=End_A(A \oplus M)$ is a finitistic Auslander algebra of finitistic dimension 2.
\end{theorem}
\begin{proof}
The condition $Ext^{1}(M,M) \neq 0$ gives us that $B$ has dominant dimension equal to two.
We have to show that every non-projective module of dominant dimension at least two has infinite projective dimension by \ref{maintheorem}. Let $N:=A \oplus M$.
This translates into the condition that every $A$-module not in $add(N)$ has infinite $add(N)$-resolution dimension. Assume there is an indecomposable module with finite $add(N)$-resolution.
Then there is a short exact sequence as follows, where the maps are minimal right $add(N)$-approximations:
$$0 \rightarrow N_1 \rightarrow N_0 \rightarrow U \rightarrow 0,$$
with $N_0,N_1 \in add(N)$ and $N_1$ being a direct sum of copies of $M$ because of the minimality. Now applying the functor $Hom(M,-)$ to this short exact sequence we obtain a long exact sequence of the form:
$$0 \rightarrow Hom(M,N_1) \rightarrow Hom(M,N_0) \rightarrow Hom(M,U) \rightarrow Ext^1(M,N_1) \rightarrow Ext^1(M,N_0) \rightarrow \cdots .$$

In this long exact sequence the map $Ext^{1}(M,N_1) \rightarrow Ext^{1}(M,N_0)$ has to be injective, since the right map in the short exact sequence is assumed to be a minimal $add(N)$-approximation which gives that $Hom(M,N_0) \rightarrow Hom(M,U)$ is surjective.
Now after removing free summands of the left map in the short exact sequence, we obtain a map $\alpha: N_1 \rightarrow N_0'$ between direct sums of copies of $M$ with the property that all components of this map are radical maps. Now $Ext^{1}(M,-)$ is a functor sending $\alpha$ to an injection. By the previous lemma this is only possible if $Ext^{1}(M,M)=0$. This contradicts our assumptions and thus there is no module with finite $add(N)$-resolution.
\end{proof}
Combining the previous result and \ref{tachtheo} we obtain our main result in this section:
\begin{theorem}
Let $A$ be a local Hopf algebra (for example a groupalgebra of a $p$-group over a field of characteristic $p$) with a nonprojective indecomposable module $M$.
Then $B:=End_A(A \oplus M)$ is a finitistic Auslander algebra of finitistic dimension 2.
\end{theorem}
\begin{proof}
By \ref{tachtheo} we have $Ext_A^1(M,M) \neq 0$ and thus the previous theorem \ref{Ricktheo} applies to give the result.

\end{proof}
The previous theorem motivates the following conjecture:
\begin{conjecture}
Let $A$ be a local selfinjective algebra and $M$ a non-projective indecomposable $A$-module. Then the algebra $B:=End_A(A \oplus M)$ is a finitistic Auslander algebra of dominant dimension 2.

\end{conjecture}
By \ref{Ricktheo}, the conjecture can equivalently states as $Ext_A^1(M,M) \neq 0$ for any non-projective module $M$ over a local selfinjective algebra.
We refer to \cite{Mar4} for more on $Ext_A^1(M,M)$ for selfinjective local algebras $A$.

\subsection{Finitistic Auslander algebras from standardly stratified algebras}
Standardly stratified algebras are a well studied class of algebras that generalise quasi-hereditary algebras.
For the basics on standardly stratified algebras we refer to \cite{Rei} and for relations to Auslander-Gorenstein algebras we refer to \cite{Mar3}. Recall that a quasi-hereditary algebra is a standardly stratified algebra with finite global dimension. This section gives the construction of standardly stratified finitistic Auslander algebras for local selfinjective algebras and special choices of modules.
Recall the following result:
\begin{theorem} \label{standstratfindimbound}
(see \cite{AHLU})
Let $A$ be a standardly stratified algebra with $n$ simple modules. Then the finitistic dimension of $A$ is bounded by $2n-2$.
\end{theorem}

Using this theorem, we can give several examples of finitistic Auslander algebras inside the class of standardly stratified algebras. We need the following result, which is the main result of \cite{CheDl}:
\begin{theorem}
\label{chedlabtheo}
 Let $A$ be a local, commutative selfinjective algebra over an algebraically closed field. Let $\mathcal{X}=(A=X(1),X(2),...,X(n))$ be a sequence of local-colocal modules (meaning that all modules have simple socle and top and therefore can be viewed as ideals of $A$) with $X(i) \subseteq X(j)$ implying $j \leq i$. Let $X= \bigoplus\limits_{i=1}^{n}{X(i)}$ and $B=End_A(X)$. Then $B$ is properly stratified with a duality iff the following two conditions are satisfied:  \newline
1. $X(i) \cap X(j)$ is generated by suitable $X(t)$ of $\mathcal{X}$ for any $1 \leq i,j \leq n$ \newline
2. $X(j) \cap \sum\limits_{t=j+1}^{n}{X(t)}=\sum\limits_{t=j+1}^{n}{X(j) \cap X(t)}$ for any $1 \leq j \leq n$.
\end{theorem}

\begin{lemma} \label{commtheorem}
Let $A$ be a commutative selfinjective algebra with an ideal $I$ with $D(I) \cong I$.
Then $B:=End_A(A \oplus I)$ is a finitistic Auslander algebra with finitistic dimension equal to 2 that is standardly stratified.
\end{lemma}
\begin{proof}
First note that $I$ being an ideal has simple socle because $A$ is selfinjective and thus has simple socle. Now $top(I) \cong top(D(I)) \cong D(soc(I))$ is again simple.
Thus \ref{chedlabtheo} applies to give that $B$ is standardly stratified. Since $B$ has two simple modules, the finitistic dimension of $B$ is bounded by 2 by \ref{standstratfindimbound}. But since $B$ is an endomorphism ring of a generator-cogenerator, its dominant dimension is at least two. Since the dominant dimension is bounded by the finitsitic dimension for non-selfinjective algebras, $B$ is a finitistic Auslander algebra with finitistic dimension equal to two. 
\end{proof}

The next result illustrates that being an Auslander-Gorenstein algebra might be extremely rare compared to the more general concept of being a finitistic Auslander algebra.    
\begin{theorem} \label{theoremexamples}
Let $A$ be an $K$-algebra.
\begin{enumerate}
\item Let $A$ be a commutative selfinjective algebra with enveloping algebra $A^{e}=A \otimes_K A$. Then $B:=End_{A^{e}}(A^{e} \oplus A)$ is a finitistic Auslander algebra of finitistic dimension equal to two that is standardly stratified.
It is an Auslander-Gorenstein algebra iff $A$ is a 2-periodic algebra iff $A \cong K[x]/(x^n)$ for some $n \geq 2$.
It is never a higher Auslander algebra.
\item Let $A$ be a selfinjective local algebra with simple module $S$. Then $End_A(A \oplus S)$ is a finitistic Auslander algebra with finitistic dimension equal to two that is standardly stratified. It is an Auslander-Gorenstein algebra iff $A \cong K[x]/(x^n)$ for some $n \geq 2$ and it is a higher Auslander algebra iff $A \cong K[x]/(x^2)$.

\end{enumerate}
\end{theorem}
\begin{proof}
\begin{enumerate}
\item Note that being commutative selfinjective implies that $A$ is even symmetric and thus $D(A) \cong A$ has $A^{e}$-bimodules. Then \ref{commtheorem} applies to show that $B$ is a finitistic Auslander algebra with finitistic dimension equal to two. Now by (1) of \ref{correspondences}, $B$ is an Auslander-Gorenstein algebra iff $\tau(A) \cong A$ as $A^{e}$-bimodules. Now since $A^{e}$ is symmetric: $\tau \cong \Omega^{2}$. And thus $B$ is an Auslander-Gorenstein algebra iff $A$ is 2-periodic iff $A \cong K[x]/(x^n)$ for some $n \geq 2$ by corollary 2.10. of \cite{Sko}. What is left to do is calculate when the algebra $B$ has finite global dimension in case $A \cong K[x]/(x^n)$.
But $B$ can have only global dimension equal to the dominant dimension equal to two iff it is a Auslander algebra iff $A^{e}$ has $A^{e}$ and $A$ as its only indecomposable modules. This is certainly never the case, since $A^{e}$ has at least two loops in its quiver and thus is never representation-finite.
\item In \cite{We} theorem 1.1., it was proven that $End_A(A \oplus S)$ is always standardly stratified in case $A$ is a selfinjective local algebra with simple module $S$. One has $Ext^{1}(S,S) \neq 0$ since the algebra is local. Thus the dominant and finitistic dimension are equal to two. Again by \cite{Sko} corollary 2.10. the module $S$ is 2-periodic iff $A \cong K[x]/(x^n)$ and the global dimension is equal to two iff it is finite iff $End_A(A \oplus S)$ is an Auslander algebra iff $A \cong K[x]/(x^2)$.
\end{enumerate}
\end{proof}

We give another example, where high dominant dimension automatically leads to being a finitistic Auslander algebra and the bound $2n-2$ for the finitistic dimension of standardly stratified algebras is attained for an arbitrary $n \geq 2$.
\begin{example}
Let $A$ be a representation-finite block of a Schur algebra with $n$ simple modules. Then $A$ has dominant dimension equal to $2n-2$. This was noted and proven in \cite{ChMar} and \cite{Mar}. By \ref{standstratfindimbound} and the fact that the dominant dimension is bounded by the finitistic dimension, it is a finitistic Auslander algebra and even a higher Auslander algebra since it has finite global dimension, being a block of a Schur algebra.
\end{example}

The next proposition shows how to construct alot of examples of finitistic Auslander algebras from known ones:
\begin{proposition}
Let $A$ be a finitistic Auslander algebra of finitistic dimension equal to $m \geq 2$ and $B$ a selfinjective algebra.
Then $A \otimes_K B$ is a finitistic Auslander algebra of finitistic dimension equal to $m$.

\end{proposition}
\begin{proof}
By \cite{ERZ} theorem 16, the finitistic dimension of the tensor product of two algebras equals the sum of the finitistic dimensions of the two algebras. Thus $A \otimes_K B$ has finitistic dimension equal to $m$, since $B$ has finitistic dimension zero as a selfinjective algebra. Now the dominant dimension of the tensor product of two algebras equals the minimum of the dominant dimension of the two algebras by \cite{Mue} lemma 6.
Thus $A \otimes_K B$ also has dominant dimension equal to $m$, since $B$ has infinite dominant dimension as a selfinjective algebra.
\end{proof}

\section{Algebras with exactly one indecomposable projective non-injective module}
We assume that all algebras in this section are basic. This is no restriction since every algebra is Morita equivalent to its basic algebra and all results and notion of this section are invariant under Morita equivalence.
For a finite dimensional algebra $A$, let $\nu_A:=DHom_A(-,A)$ denote the Nakayama functor. Following \cite{FHK}, an algebra $A$ is called \emph{almost selfinjective} in case there is at most one indecomposable projective non-injective module and the injective envelope $I(A)$ of the regular module $A$ has the property that $\nu_A^i(I(A))$ is projective for all $i \geq 0$.
We define a \emph{generalised almost selfinjective algebra} as an algebra with dominant dimension at least one that has at most one indecomposable projective non-injective module. Those algebras generalise the almost selfinjective algebras from \cite{FHK}, see lemma 2.6. in \cite{FHK}. 
Again we are mainly interested in the non-selfinjective generalised almost selfinjective algebras. We call non-selfinjective generalised almost selfinjective algebras for short \emph{NGAS algebras} in the following.
We first show how to construct large classes of NGAS algebras.
We recall the construction of SGC-extension algebras from \cite{CIM}:
\begin{definition}
Let $A$ be a non-selfinjective algebra. For $m \geq 1$ define the \emph{$m$-th SGC-extension algebra} of $A$ as $A^{(m+1)}:=End_{A^{(m)}}(N_m)$, when $A^{(0)}=A$ and $N_m$ denotes the basic version of the module $A^{(m)} \oplus D(A^{(m)})$.
\end{definition}

The next proposition provides us with the construction of large classes of NGAS algebras:
\begin{proposition}
\begin{enumerate}
\item Let $A$ be an algebra with exactly one indecomposable projective non-injective module. Then the $m$-th SGC-extension algebra of $A$ is a NGAS algebra for each $m \geq 1$.
\item Let $A$ be a selfinjective algebra and $M$ an indecomposable non-projective $A$-module. Then $B:=End_A(A \oplus M)$ is a NGAS algebras.
\end{enumerate}
\end{proposition}
\begin{proof}
\begin{enumerate}
\item 
Let $A$ be an algebra with exactly one indecomposable projective non-injective module and $B$ the SGC-extension of $A$. Assume $A$ has $n$ simple modules.
Recall from \ref{ARSmaintheorem} (1) that the number of indecomposable projective-injective $B$-modules equals the number of indecomposable injective $A$-modules, which is $n$. The number of simple $B$-modules equals $n+1$ by assumption on $A$.
Since $B$ is isomorphic to the endomorphism ring of a generator-cogenerator $B$ has dominant dimension at least two and thus is a NGAS algebra.
\item As before, from \ref{ARSmaintheorem} (1) it follows that $B$ has $n+1$ simple modules and $n$ indecomposable projective-injective modules and dominant dimension at least two when $A$ has $n$ simple modules.
\end{enumerate}
\end{proof}
(1) of the previous proposition applies for example to every local non-selfinjective algebra $A$ to provide infinitely many NGAS algebras $A^{(m)}$ for $m \geq 1$. The previous proposition also shows that NGAS algebras are closed under SGC-extensions.
Let $A$ be an algebra of dominant dimension $d \geq 1$ with minimal faithful projective-injective module $eA$.
Following \cite{PS}, we the basic version of the module of the form $eA \oplus \Omega^{-i}(A)$ a \emph{shifted tilting module} for each $0 \leq i \leq d$.
That those modules are really tilting modules can be found in \cite{PS}. Dually, the basic version of the modules of the form $eA \oplus \Omega^i(D(A))$ are called \emph{coshifted cotilting modules} for $0 \leq i \leq d$.
For our next results, we recall several results from \cite{BS}.
Let $M$ be a module and 
$$0 \rightarrow A \xrightarrow{f_1} M_1 \xrightarrow{f_2} \cdots \xrightarrow{f_n} M_n \cdots$$
a complex with $K_i = cokern(f_i)$ for $i \geq 1$ and $K_0=A$ such that each map $K_i \rightarrow M_{i+1}$ is a minimal left $add(M)$-approximation. Let $\eta_n$ be the truncated complex ending in $M_n$ obtained from the above complex.
Then $M$ is said to have \emph{faithful dimension} $fadim(M)$ equal to $n$ if $\eta_n$ is exact, but $\eta_{n+1}$ is not.
For a module $M$ let $\delta(M)$ denote the number of indecomposable non-isomorphic direct summands of $M$.
Recall that an \emph{almost complete cotilting module} $M$ over an algebra $A$ is a direct summand of a cotilting module $N$ such that $\delta(M)=\delta(A)-1$. A \emph{complement} of an almost complete cotilting module $M$ is a module $X$ such that $M \oplus X$ is a cotilting module.
\begin{theorem} \label{BStheorem}
Let $M$ be an $A$-module and $C:=End_A(M)$. Note that we can view $M$ as a left $C$-module.
\begin{enumerate}
\item $M$ has faithful dimension at least two if and only if $A \cong End_C(M)$.
\item In case $M$ has faithful dimension at least two, we have $fadim(M)=n$ if and only if $Ext_C^i(M,M)=0$ for $i=1,...,n-2$ and $Ext_C^{n-1}(M,M) \neq 0$.
\item Let $M$ be an almost complete cotilting module. Then $M$ has exactly $n+1$ indecomposable complements if and only if $fadim(M)=n$.
\item Let $M$ be an almost complete cotilting module that is faithful. Then $M$ has at least two complements.
\end{enumerate}
\end{theorem}
\begin{proof}
\begin{enumerate}
\item See proposition 2.1. of \cite{BS}.
\item See proposition 2.2. of \cite{BS}.
\item See theorem 3.6. of \cite{BS}.
\item See theorem 3.1. of \cite{BS}.
\end{enumerate}
\end{proof}
The next proposition gives a nice characterisation of NGAS algebras using such shifted tilting modules and coshifted cotilting modules. 
\begin{proposition} \label{propocotilting}
Let $A$ be an algebra with finite dominant dimension $n \geq 1$.
The the following are equivalent:
\begin{enumerate}
\item $A$ is a NGAS algebra.
\item $A$ has exactly $n+1$ basic cotilting modules.
\item $A$ has exactly $n+1$ basic tilting modules.
\end{enumerate}
All basic cotilting modules are in this case isomorphic to the basic version of $\Omega^{i}(D(A)) \oplus eA$ and all basic tilting modules are in this case isomorphic to the basic version of $\Omega^{-i}(A) \oplus eA$ for some $0 \leq i \leq n$.
\end{proposition}
\begin{proof}
We show the equivalence of (1) and (2). The equivalence of (1) and (3) is shown similar.
Assume first that $A$ has exactly one indecomposable projective non-injective module and let $eA$ be the minimal faithful projective-injective $A$-module. By the definition of the faithful dimension, the faithful dimension of $eA$ is equal to the dominant dimension of $A$. Since $eA$ is projective-injective, every cotilting module has $eA$ as a direct summand and we have $\delta(eA)=\delta(A)-1$ by assumption on $A$.
Then it follows from (3) of \ref{BStheorem} that $eA$ has exactly $n+1$ complements.  Now the coshifted cotilting modules $\Omega^{i}(D(A)) \oplus eA$ are $n+1$ non-isomorphic cotilting modules and thus the proposition follows.

Assume now that $A$ has exactly $n+1$ cotilting modules. Since $A$ has $n+1$ coshifted cotilting modules, this means that every cotilting module must be a coshifted cotilting module. We show that this forces $A$ to have exactly one indecomposable projective non-injective module. 
Assume otherwise, that is $A$ has more than one indecomposable projective non-injective module. Since $A$ is assumed to have positive dominant dimension, there exists a minimal faithful projective-injective module $eA$. Then $A$ also has two indecomposable injective non-projective modules. Let $I_1$ and $I_2$ be two non-isomorphic indecomposable projective non-injective modules. Let $M:=D(A)/I_1$ and note that $M$ has $eA$ as a direct summand. Thus $M$ is faithful and an almost complete cotilting module since $M \oplus I_1=D(A)$ is a cotilting module.
By \ref{BStheorem} (4), there exists at least one indecomposable module $X$ that is not isomorphic to $I_1$ such that the module $T:=M \oplus X$ is a cotilting module.
We show that $T$ can not be a coshifted cotilting module.
Let $\delta(A)=r$. Since we assume that $A$ is not a NASG algebra, there are $s$ indecomposable projective-injective module with $s \leq r-2$.
Just note that $T$ is not isomorphic to $D(A)$ and that $T$ has $r-1$ indecomposable injective direct sumannds while every coshifted cotilting module that is not injective has exactly $s \leq r-2$ indecomposable injective direct summands. Thus $T$ is a colting module that is not isomorphic to any of the coshifted cotilting modules and $A$ has therefore more than $n+1$ cotilting modules. This is a contradiction and thus $A$ has to be a NASG algebra.

\end{proof}
We define the \emph{finitistic injective dimension} of an algebra as the supremum of injective dimensions of modules having finite injective dimension. Note that in general the finitistic dimension does not coincide with the finitistic injective dimension. We call an algebra a \emph{weak finitistic Auslander algebra} in case it has positive dominant dimension equal to the finitistic dimension. We call an algebra a \emph{weak co-finitistic Auslander algebra} in case the dominant dimension equals the finitistic injective dimension of the algebra (note that the definition is really dual, since the dominant dimension of an algebra always coincides with the codominant dimension of an algebra).
We do not know whether every weak finitistic Auslander algebra is a weak co-finitistic Auslander algebra. In fact the author is not even aware of an algebra that has dominant dimension at least one such that the finitistic dimension is not equal to the finitistic injective dimension.
We call an algebra \emph{tilting-finitistic} in the following in case the finitistic dimension of $A$ equals the supremum of projective dimension of tilting modules. We can apply \ref{propocotilting} to obtain the following theorem.

\begin{theorem} \label{nasgtheo}
Let $A$ be a NASG algebra with finite finitistic dimension. Then $A$ is a weak finitistic Auslander algebra in case $A$ is tilting finitistic.

\end{theorem}
\begin{proof}
Since $A$ is assumed to be tilting finitistic and a NASG algebra, the finitistic dimension of $A$ equals the supremum of the shifted tilting modules which is equal to $d$ when $d$ denotes the dominant dimension of $A$ because the module $eA \oplus \Omega^{-d}(A)$ is in this case the tilting module of highest projective dimension. 
\end{proof}
This motivates the following question:
\begin{question}
Is every finite dimensional algebra tilting-finitistic?
\end{question}
This question is for example studied in \cite{AT} and it seems to be open in general.
We now obtain several corollaries.
\begin{corollary} \label{nasgcorollary}
Let $A$ be a NASG algebra. Then $A$ is a finitistic Auslander algebra in each of the following cases:
\begin{enumerate}
\item The full subcategory of modules having finite projective dimension is contravariantly finite.
\item $A$ is representation-finite.
\item $A$ is Gorenstein.
\end{enumerate}
\end{corollary}
\begin{proof}
We show that $A$ is tilting-finitistic in each case:
\begin{enumerate}
\item This follows from \cite{AT}, theorem 2.6.
\item This follows from (1), since every subcategory of the module category of a representation-finite algebra is contravariantly finite.
\item Recall that for Gorenstein algebras, the finitistic dimension equals the Gorenstein dimension. Then the module $D(A)$ is a tilting module of projective dimension equal to the Gorenstein dimension.
\end{enumerate}

\end{proof}
The large class of examples motivate the following conjecture:
\begin{conjecture} \label{nasgconjecture}
Let $A$ be a NASG algebra. 
Then $A$ is tilting finitistic and thus a weak finitistic Auslander algebra 
\end{conjecture}

\begin{corollary}
Let $A$ be a representation-finite NASG algebra. Then $A$ is a weak finitistic Auslander algebra and a weak finitistic co-Auslander algebra.
\end{corollary}
\begin{proof}
This follows directly from \ref{nasgcorollary} and \ref{nasgtheo} and their duals.
\end{proof}

We also obtain a corollary for higher Auslander algebras and Auslander-Gorenstein algebras that seems to be new:
\begin{corollary}
Let $A$ be a NASG algebra of dominant dimension at least two.
\begin{enumerate}
\item $A$ is an Auslander-Gorenstein algebra if and only if $A$ is Gorenstein.
\item $A$ is a higher Auslander algebra if and only if $A$ has finite global dimension.
\end{enumerate}

\end{corollary}

\begin{proof}
\begin{enumerate}
\item If $A$ is Auslander-Gorenstein, then $A$ is by definition also Gorenstein.
Assume that $A$ is Gorenstein and a NASG algebra. Then by \ref{nasgcorollary} (3), it is a finitistic Auslander algebra and thus an Auslander-Gorenstein algebra because it has finite Gorenstein dimension.
\item If $A$ is a higher Auslander algebra, then $A$ is has by definition finite global dimension.
Assume that $A$ has finite global dimension and is a NASG algebra. Then by \ref{nasgcorollary} (3) (recall that every algebra of finite global dimension is a Gorenstein algebra), it is a finitistic Auslander algebra and thus a higher Auslander algebra because it has finite global dimension.
\end{enumerate}
\end{proof}

We give a large class of concerete examples of representation-finite NASG algebras:
\begin{example}
Let $A$ be a Nakayama algebra. Then $A$ always has dominant dimension at least one and thus is a NASG algebra if an only if it has exactly one indecomposable projective non-injective module. This is the case if and only if $A$ has Kupisch series (see for example \cite{Mar2} for background on Nakayama algebras and their Kupisch series) equal to $[2,2,...,2,1]$ or $[a,a,...,a,a+1,a+1,...,a+1]$ for some natural number $a \geq 2$. Thus for any number $n \geq 2$ we get like this infinitely many representation-finite weak finitistic Auslander algebras having $n$ simple modules.
Proving directly from the definition of the finitistic dimension and dominant dimension that those algebras are weak finitistic Auslander algebras seems very hard and the author is not aware of a direct proof.
Computer experiments suggest that every non-selfinjective Nakayama algebra with exactly $n \geq 2$ simple modules and dominant dimension at least $n$ is automatically a NASG algebra and thus also a weak finitistic Auslander algebra.
This is proved with a computer for $n \leq 13$ and we will give more results on this problem in forthcoming work.

\end{example}

The previous example and all other example which the author considered lead to the following conjecture:
\begin{conjecture}
Let $n \geq 2$.
There exists a polynomial function $f(n)$ such that the following is true: \newline
Every connected non-selfinjective algebra with $n$ simple modules that has dominant dimension at least $f(n)$ is a finitistic Auslander algebra.
\end{conjecture}
We think that a polynomial function $f(n)$ with $n \leq f(n) \leq 2n$ might do the job. One can verify this conjecture for several classes of algebras, as we will do in forthcoming work. We give one more example.
\begin{example}
In \cite{ChMar} the class of representation-finite gendo-symmetric biserial algebras was classified (generalising the classical Brauer tree algebras). All those algebras were Gorenstein and thus their finitistic dimension coincides with the Gorenstein dimension. Explicit values for the dominant and Gorenstein dimension are obtained and one can easily show that the above conjecture is true for this class of algebras with $f(n)=n$.
\end{example}

\end{document}